\documentclass[12pt]{amsart}

\usepackage[all]{xy}
\usepackage{fullpage}
\usepackage{amsmath}
\usepackage{amssymb}
\usepackage{amsthm}
\usepackage{mathrsfs}
\usepackage{multicol}

\theoremstyle{plain}
\newtheorem{thm}{Theorem}
\newtheorem{lem}[thm]{Lemma}
\newtheorem{cor}[thm]{Corollary}

\theoremstyle{definition}
\newtheorem{defn}[thm]{Definition}
\newtheorem{ex}[thm]{Example}

\title{A description of the Parry-Sullivan number of a graph using circuits}
\author{Chris Smith}
\address{Department of Mathematics, University of Colorado, Colorado Springs CO 80933 U.S.A.}
\email{cdsmith@gmail.com}

\begin{document}

\maketitle

\begin{abstract}
In this short note, we give a description of the Parry-Sullivan number of a graph in terms of
the cycles in the graph.  This tool is occasionally useful in reasoning about the
Parry-Sullivan numbers of graphs.
\end{abstract}

Given a graph $E$ with $n$ vertices, one may define the incidence matrix $A_E$ as the
$n \times n$ matrix wherein each entry $(A_E)_{ij}$ is defined to be the number of edges in
$E$ from the $i$th vertex to the $j$th vertex.  Parry and Sullivan showed the quantity
$\det(I-A_E)$, now known as the Parry-Sullivan number of the graph and denoted
$\mathcal{PS}(E)$, is an invariant of the flow equivalence class of the subshift of finite
type induced by $E$ \cite{ParrySullivan}.  (Equivalently, some sources choose to view
$\mathcal{PS}(E)$ as $\det(I-A_E^t)$.)  In working with this invariant, it is sometimes
convenient to view the Parry-Sullivan number of a graph in terms of the structure of cycles
of the graph, rather than as a determinant calculation.  We establish such a characterization
here.

Some basic definitions are given here and used throughout this note.

\begin{defn}

A \textbf{graph} $E$ (also known as a directed graph) is a 4-tuple $(E^0, E^1, r, s)$, where
$E^0$ is a set of vertices, $E^1$ is a set of edges, and $r, s : E^1 \to E^0$ associate each
edge with its range and source, respectively.

\begin{itemize}
\item A \textbf{directed cycle} of a graph $E$ is a sequence of edges $e_1 e_2 \cdots e_m$ such that
$s(e_1) = r(e_m)$, and $s(e_i) \neq s(e_j)$ whenever $i \neq j$.  In particular, the choice of
starting vertex, $s(e_1)$, distinguishes two cycles that follow the same edges.

\item A \textbf{directed circuit} in a graph $E$ is a finite set of edges $C \subseteq E^1$ with the
property that the edges in $C$ can be arranged into a directed cycle.

\item Let $\mathcal{C}$ be a set of directed circuits.  $\mathcal{C}$ is \textbf{vertex-disjoint} if
for any $C_1, C_2 \in \mathcal{C}$, the sets $\{ s(e) | e \in C_1 \}$ and $\{ s(e) | e \in C_2 \}$
are disjoint.
\end{itemize}
\end{defn}

In this note, by the words cycle and circuit, we always mean directed cycles and directed circuits.
We now view circuits as inducing a permutation on the vertices of a graph, in the obvious way:

\begin{defn}
Let $\mathcal{C}$ be a set of vertex-disjoint circuits of a graph $E$.  The \textbf{permutation
$\rho_{\mathcal{C}} : E^0 \to E^0$ induced by $\mathcal{C}$} is defined as follows.  If there
exists $e \in \bigcup \mathcal{C}$ such that $s(e) = v$, then $\rho_{\mathcal{C}}(v) = r(e)$.
Otherwise, $\rho_{\mathcal{C}}(v) = v$.
\end{defn}

Then $\rho_\mathcal{C}$ is well-defined because $\mathcal{C}$ is vertex-disjoint, and it is easily
seen to be a permutation of $E^0$.  The curcuits having more than one edge in $\mathcal{C}$ correspond
with cycles in the representation of $\rho_{\mathcal{C}}$ as a product of disjoint cycles.

Several distinct such vertex-disjoint sets of curcuits may induce the same permutation of $E^0$.  The
induced permutation $\rho_{\mathcal{C}}$ can be viewed as defining an equivalence relation on
vertex-disjoint sets of circuits, as follows:

\begin{defn}
The equivalence relation $\sim_\rho$ is defined such that $\mathcal{C}_1 \sim_\rho \mathcal{C}_2$
if and only if $\rho_{\mathcal{C}_1} = \rho_{\mathcal{C}_2}$.
\end{defn}

A quick lemma about permutations will be useful later.

\begin{lem}
Let $\rho \in S_n$.  If $\rho$ is the product of $m$ disjoint cycles
$\sigma_1\sigma_2\cdots\sigma_m$, and $\rho$ has $k$ fixed points, then
$\rho$ is a product of $n - (m + k)$ transpositions.
\end{lem}
\begin{proof}
A cycle of length $j$ can be obtained by a product of $j-1$ transpositions.  By
repeating this for each disjoint cycle in $\rho$, one can express $\rho$ as the
product of $\sum_{i=1}^m (|\sigma_i| - 1)$ transpositions.  But each
$1 \leq a \leq n$ is either a member of some $\sigma_i$ or is a fixed point
of $\rho$, so $\sum_{i=1}^m |\sigma_i| = n - k$, and rearranging the sum gives
$\rho$ as a product of
\[
\sum_{i=1}^m |\sigma_i| - \sum_{i=1}^m 1 = n - (m + k)
\]
transpositions, as desired.
\end{proof}

Given some finite graph $E$, the main result in this note is that
$\mathcal{PS}(E)$ is equal to the number of even-sized vertex-disjoint sets of circuits minus the
number of odd-sized vertex-disjoint sets of circuits in $E$.  We proceed by counting this value
within each equivalence class with respect to $\sim_\rho$ separately, and then adding them.

\begin{lem}
Let $E$ be a finite directed graph, and $\rho$ be a permutation of $E^0$.  Furthermore, let:
\begin{itemize}
\item $\mathscr{C}_\rho$ be the set of all vertex-disjoint sets of circuits that induce the permutation $\rho$.
\item $\mathscr{C}_\rho^e$ be those elements of $\mathscr{C}_\rho$ having an even number of cycles.
\item $\mathscr{C}_\rho^o$ be those elements of $\mathscr{C}_\rho$ having an odd number of cycles.
\end{itemize}
The signed elementary product of $I - A_E$ corresponding to the permutation $\rho$ is equal to
$|\mathscr{C}_\rho^e| - |\mathscr{C}_\rho^o|$.
\end{lem}
\begin{proof}
At the top-level, we proceed by induction on the number of cycles in the expression of
$\rho$ as a product of disjoint cycles.  There are two base cases, followed by the
inductive case.

\textbf{Base case 1: Identity permutation.}  The first base case for the induction
is the case in which $\rho$ has no disjoint cycles.  In other words, $\rho$ is the
identity permutation on the vertices of $E$.  Note first of all that the identity
permutation is always even, so the signed elementary product is
\[\prod_{i=1}^n (I - A_E)_{ii} = \prod_{i=1}^n (1 - (A_E)_{ii})\]

We now perform a sub-induction by the number of vertices in $E$.  Suppose $E$ has only one
vertex.  Then $\mathscr{C}_\rho^e$ is precisely $\{ \emptyset \}$, and $\mathscr{C}_\rho^o$ is
precisely $\{ \{ e \} | \textrm{$e$ is a loop in $E$} \}$.  There are $(A_E)_{11}$ loops at the
sole vertex of $E$, so
\[|\mathscr{C}_\rho^e| - |\mathscr{C}_\rho^o| = 1 - (A_E)_{11}\]

For the inductive step of this sub-induction, we assume from the inductive hypothesis that
the conclusion is true for $\rho$ as the identity and $E$ having $n$ vertices.  Now
suppose $E$ has $n+1$ vertices.  Then $I - A_E$ looks like
\[
\left[\begin{array}{ccc|c}
1 - (A_E)_{11} &        &                &         \\
               & \ddots &                &  \vdots \\
               &        & 1 - (A_E)_{nn} &         \\
\hline
               & \cdots &                & 1 - (A_E)_{n+1,n+1}
\end{array}\right]
\]

Consider the subgraph $F$ of $E$ obtained by dropping vertex $n+1$ and its adjacent
edges.  Then $\rho$ restricted to the vertices of $F$ is still an identity
permutation, and $F$ has $n$ vertices.  Therefore, the inductive hypothesis can be
applied, yielding
\[|\mathscr{C}_{\rho|_F}^e| - |\mathscr{C}_{\rho|_F}^o|
  = \prod_{i=1}^n (1 - (A_F)_{ii})
  = \prod_{i=1}^n (1 - (A_E)_{ii})\]

Each vertex-disjoint set of circuits in $F$ can be extended to $E$ in two ways.  It can be
left as is, or it may be extended with a single loop at vertex $n+1$.  Furthermore,
since sets of circuits that induce the identity permutation may only contain loops, any
vertex-disjoint set of circuits that induces the identity permutation can be obtained in
this way from some such set in $F$, and that subset will induce the identity permutation
in $F$.  Therefore, each vertex-disjoint set of circuits in $F$ that induces the identity
permutation corresponds to one such set of circuits in $E$ with the same parity (obtained
by simply treating the existing set as a set in $E$), and $(A_E)_{n+1,n+1}$ sets of opposite
parity (obtained by adding a loop at vertex $n+1$).  Therefore,
\[
|\mathscr{C}_\rho^e| - |\mathscr{C}_\rho^o|
  = (1 - (A_E)_{n+1,n+1}) \prod_{i=1}^n (1 - (A_E)_{ii})
  = \prod_{i=1}^{n+1} (1 - (A_E)_{ii})
\]
This is the elementary product of $I - A_E$ corresponding to the identity permutation, so
this completes the induction.

\textbf{Base case 2: $n$-Cycle permutation.}  The second base case for the
induction is for a permutation $\rho$ consisting of a single cycle that contains all
vertices of the graph.  Note that in this case, the permutation fixes no vertices and
can be built from $n-1$ transpositions, so the signed elementary product is
\[
    (-1)^{n-1} \prod_{i=1}^n -(A_E)_{i,\rho(i)}
  = (-1)^{n-1} (-1)^n \prod_{i=1}^n (A_E)_{i,\rho(i)}
  = -\prod_{i=1}^n (A_E)_{i,\rho(i)}
\]
since exactly one of $n$ or $n-1$ will be even.

A vertex-disjoint set of circuits of $E$ that induces an $n$-cycle permutation must
contain only one circuit, which includes all vertices of $E$.  The number of ways to choose
such a cycle is precisely $\prod_{i=1}^n (A_E)_{i,\rho(i)}$, and they will all have size
$1$, which is an odd number. Therefore, the number of even-sized vertex-disjoint set of
circuits that induce $\rho$ minus the number of odd-sized such sets is
\[ 0 - \prod_{i=1}^n (A_E)_{i,\rho(i)} \]
which is precisely the same as the signed elementary product.

\textbf{Inductive case.} Now suppose $\rho$ is not the identity permutation,
and does not consist of a single $n$-cycle.  Then there exists, in the expression of
$\rho$ as a product of disjoint cycles, a cycle $\sigma$ that does not include all
vertices of $E$.  Let $F$ be the subgraph of $E$ obtained by deleting the vertices
of $\sigma$ and all adjacent edges.  Note that $\rho|_F$ is a permutation whose
expression as a product of disjoint cycles contains one fewer cycle than $\rho$.
Therefore, the inductive hypothesis applies to $\rho|_F$, giving that
$|\mathscr{C}_{\rho|_F}^e| - |\mathscr{C}_{\rho|_F}^o|$ is the signed elementary
product of $A_F$ corresponding to $\rho|_F$.

Each vertex disjoint set of circuits in  $\mathscr{C}_{\rho|_F}$ can be extended to
a vertex-disjoint set of circuits in $\mathscr{C}_{\rho}$ by adding a circuit that
induces $\sigma$.  Rearranging the vertices if needed so that $\sigma$ moves vertices
$n-|\sigma|+1$ through $n$, the number of ways of doing this is
$\prod_{i = n-|\sigma|+1}^n (A_E)_{i,\rho(i)}$.  Additionally, adding one new circuit
changes the parity of the size of each vertex-disjoint set of circuits, so that
\[
|\mathscr{C}_{\rho}^e| - |\mathscr{C}_{\rho}^o|
= - \left(\prod_{i = n-|\sigma|+1}^n (A_E)_{i,\rho(i)}\right)
    \left(|\mathscr{C}_{\rho|_F}^e| - |\mathscr{C}_{\rho|_F}^o|\right)
\]

Applying the conclusion above from the inductive hypothesis,
\[
|\mathscr{C}_{\rho}^e| - |\mathscr{C}_{\rho}^o|
= - \left(\prod_{i = n-|\sigma|+1}^n (A_E)_{i,\rho(i)}\right)
    \left(\mathrm{sgn}(\rho|_F) \prod_{i=1}^{n-|\sigma|} (I - A_E)_{i,\rho(i)}\right)
\]

Note that $\rho$ does not fix any of the vertices from $n - |\sigma| + 1$ to $n$, since
they are part of the cycle $\sigma$.  Therefore, $(I - A_E)_{i,\rho(i)} = -(A_E)_{i,\rho(i)}$
for $i$ in that range.  Therefore, this can simplified to
\[
|\mathscr{C}_{\rho}^e| - |\mathscr{C}_{\rho}^o|
= \mathrm{sgn}(\rho|_F) (-1)^{|\sigma| + 1} \left(\prod_{i=1}^n (I - A_E)_{i,\rho(i)}\right)
\]

This has the same magnitude as the signed elementary product.  To show that the sign is
correct, recall from the earlier lemma that $\rho$ is a product of $n - (k + m)$ transpositions,
where $k$ is the number of fixed points of $\rho$, and $m$ is the number of disjoint cycles.
Similarly, $\rho|_F$ has $|\sigma|$ fewer vertices than $\rho$, the same number of fixed points,
and one fewer cycle, so it is a product of $n - |\sigma| - (k + (m - 1))$ transpositions.
Therefore, the number of transpositions differ by $|\sigma| - 1$, which has the same parity as
the $|\sigma| + 1$ above.

Applying induction a final time, we conclude that the lemma holds for all $\rho$.
\end{proof}

We've reached the main result.

\begin{thm}\label{mainthm}
Let $E$ be a finite directed graph, and $A_E$ be its adjacency matrix.  Then $\mathcal{PS}(E)$
is the number of even-sized vertex-disjoint sets of circuits of $E$ minus the number of
odd-sized vertex-disjoint sets of circuits of $E$.
\end{thm}
\begin{proof}
For each $\rho \in S_n$, the signed elementary product corresponding to $\rho$ is equal
to $|\mathscr{C}_{\rho}^e| - |\mathscr{C}_{\rho}^o|$, which is the desired number for
those vertex-disjoint sets of circuits within the equivalence class of $\sim_\rho$ that
induces $\rho$.  Since each vertex-disjoint set of circuits induces some permutation, adding
the signed elementary products for all $\rho \in S_n$ gives both the desired subtraction and
the $\det(I-A_E) = \mathcal{PS}(E)$.
\end{proof}

We note, as an intuitive application of this result, that the information contained in the
Parry-Sullivan number of a graph relates to the manner in which the circuits of the graph share
vertices.  In particular, suppose a graph $E$ contains $n$ distinct circuits, and consider all
of the possible sets of circuits of $E$, without regard to whether they are vertex-disjoint.
As an easy application of the binomial theorem, the number of even-sized sets of circuits minus
the number of odd-sized sets of circuits is equal to
\[ \sum_{i=0}^n \binom{n}{i} (-1)^i = (1-1)^n = 0. \]
Therefore, when the Parry-Sullivan number of a graph is non-zero, it is because certain of these
circuits shared vertices.  The Parry-Sullivan number can be viewed as a measure of the manner in
which this sharing of vertices occurs.

\begin{ex}
Suppose $E$ is the graph
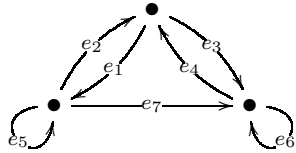
\[\xymatrix{
&  \bullet \ar@/^/[dl]|{e_1} \ar@/^/[dr]|{e_3}               \\
   \bullet   \ar@/^/[ur]|{e_2} \ar@(l,d)|{e_5}  \ar[rr]|{e_7}
&& \bullet   \ar@/^/[ul]|{e_4} \ar@(r,d)|{e_6}               \\
}\]

The Parry Sullivan number of $E$ can be computed as \[ \mathcal{PS}(E) = \det(I-A_E) =
\det \left(\begin{array}{rrr} 1 & -1 & -1 \\ -1 & 0 & -1 \\ -1 & 0 & 0
\end{array}\right) = -1.\]  Alternatively, the same number can be computed via theorem
\ref{mainthm}.  Although circuits are sets of edges, we find it less cumbersome to
write them as cycles, simply keeping in mind that they are not distinguished by the
starting vertex.  Then the even-sized vertex-disjoint sets of circuits are: $\emptyset$,
$\{ e_5, e_3e_4 \}$, $\{ e_6, e_1e_2 \}$, and $\{ e_5, e_6 \}$.  The odd vertex-disjoint
sets of circuits are: $\{ e_1e_2 \}$, $\{ e_3e_4 \}$, $\{ e_5 \}$, $\{ e_6 \}$, and $\{ e_1e_7e_4 \}$.
Therefore, $\mathcal{PS}(E) = 4 - 5 = -1$.
\end{ex}

As one easy consequence of this result, we conclude by re-establishing a (well-known) fact about
the Parry-Sullivan number of a graph: namely, that the elimination of sources and sinks from the
graph does not change the Parry-Sullivan number.

\begin{cor}
Let $E$ be a graph in which $v$ is a source or a sink.  Let $E_{\backslash v}$ be the graph
obtained by deleting $v$ and any adjacent edges from $E$.  Then the Parry-Sullivan number
of $E$ and $E_{\backslash v}$ are the same.
\end{cor}
\begin{proof}
Suppose $\mathcal{C}$ is a vertex-disjoint set of circuits in $E$.  Then no circuit in
$\mathcal{C}$ may contain any edge adjacent to $v$, because $v$ does not lie in a cycle.
Therefore, $\mathcal{C}$ is also a vertex-disjoint set of circuits in $E_{\backslash v}$.
Conversely, if $\mathcal{C}$ is a vertex-disjoint set of circuits in $E_{\backslash v}$,
then every edge of $E_{\backslash v}$ also occurs in $E$, and so $\mathcal{C}$ is also a
vertex-disjoint set of circuits in $E$.  Since $E$ and $E_{\backslash v}$ have the same
vertex-disjoint sets of circuits, their Parry-Sullivan numbers are the same.
\end{proof}

\section*{Acknowledgements}

The author wishes to thank Gene Abrams for his invaluable suggestions and careful reading during the
preparation of this note.

\end{document}